\documentclass{amsart}
\usepackage{amssymb,amsmath,amsthm,amsfonts}
\usepackage {latexsym,enumerate}
\usepackage{bbm}

\allowdisplaybreaks
\newcommand{\nc}{\newcommand}
\nc{\les}{\lesssim}
\nc{\nit}{\noindent}
\nc{\nn}{\nonumber}
\nc{\D}{\partial}
\nc{\diff}[2]{\frac{d #1}{d #2}}
\nc{\diffn}[3]{\frac{d^{#3} #1}{d {#2}^{#3}}}
\nc{\pdiff}[2]{\frac{\partial #1}{\partial #2}}
\nc{\pdiffn}[3]{\frac{\partial^{#3} #1}{\partial{#2}^{#3}}}
\nc{\abs}[1] {\lvert #1 \rvert}
\nc{\cAc}{{\cal A}_c}
\nc{\cE}{{\cal E}}
\nc{\cF}{{\cal F}}
\nc{\cP}{{\cal P}}
\nc{\cV}{{\cal V}}
\nc{\cQ}{{\cal Q}}
\nc{\cGin}{{\cal G}_{\rm in}}
\nc{\cGout}{{\cal G}_{\rm out}}
\nc{\cO}{{\cal O}}
\nc{\Lav}{{\cal L}_{\rm av}}
\nc{\cL}{{\cal L}}
\nc{\cB}{{\cal B}}
\nc{\cZ}{{\cal Z}}
\nc{\cR}{{\cal R}}
\nc{\cT}{{\cal T}}
\nc{\cY}{{\cal Y}}
\nc{\cX}{{\cal X}}
\nc{\cXT}{{{\cal X}(T)}}
\nc{\cBT}{{{\cal B}(T)}}
\nc{\vD}{{\vec \mathcal{D}}}
\nc{\efield}{\mathcal{E}}
\nc{\vE}{{\vec \efield}}
\nc{\vB}{{\vec \mathcal{B}}}
\nc{\vH}{{\vec \mathcal{H}}}
\nc{\mR}{\mathcal R}
\nc{\mF}{\mathcal F}
\nc{\mE}{\mathcal E}
\nc{\ty}{{\tilde y}}
\nc{\tu}{{\tilde u}}
\nc{\tV}{{\tilde V}}
\nc{\Pc}{{\bf P_c}}
\nc{\bx}{{\bf x}}
\nc{\bX}{{\bf X}}
\nc{\bXYZ}{{\bf XYZ}}
\nc{\bY}{{\bf Y}}
\nc{\bF}{{\bf F}}
\nc{\bS}{{\bf S}}
\nc{\dV}{{\delta V}}
\nc{\dE}{{\delta E}}
\nc{\TT}{{\Theta}}
\nc{\dPsi}{{\delta\Psi}}
\nc{\order}{{\cal O}}
\nc{\Rout}{R_{\rm out}}
\nc{\eplus}{e_+}
\nc{\eminus}{e_-}
\nc{\epm}{e_\pm}
\nc{\sgn}{\text{sgn}}
\nc{\eps}{\varepsilon}
\nc{\vnabla}{{\vec\nabla}}
\nc{\G}{\Gamma}
\nc{\w}{\omega}
\nc{\mh}{h}
\nc{\mg}{g}
\nc{\vphi}{\varphi}
\nc{\tlambda}{\tilde\lambda}
\nc{\be}{\begin{equation}}
\nc{\ee}{\end{equation}}
\nc{\ba}{\begin{eqnarray}}
\nc{\ea}{\end{eqnarray}}

\nc{\g}{\gamma}
\nc{\ol}{\overline}

\newtheorem{theorem}{Theorem}[section]
\newtheorem{lemma}[theorem]{Lemma}
\newtheorem{prop}[theorem]{Proposition}
\newtheorem{corollary}[theorem]{Corollary}

\newtheorem{rmk}[theorem]{Remark}
\newtheorem{asmp}[theorem]{Assumption}

\nc{\pT}{\partial_T}
\nc{\pz}{\partial_z}
\nc{\pt}{\partial_t}
\nc{\la}{\langle}
\nc{\ra}{\rangle}
\nc{\infint}{\int_{-\infty}^{\infty}}
\nc{\halfwidth}{6.5cm}
\nc{\figwidth}{10cm}
\newcommand{\f}{\frac}

\nc{\nlayers}{L} \nc{\nsectors}{M}
\nc{\indicator}{\mathbf{1}}
\nc{\Rhole}{R_{\rm hole}}
\nc{\Rring}{R_{\rm ring}}
\nc{\neff}{n_{\rm eff}}
\nc{\Frem}{F_{\rm rem}}
\nc{\R}{\mathbb R}
\nc{\mJ}{\mathcal J}
\nc{\C}{\mathbb C}
\nc{\Z}{\mathbb Z}
\nc{\N}{\mathbb N}
\nc{\DD}{\Delta}
\nc{\cD}{\mathcal D}
\nc{\lnorm}{\left\|}
\nc{\rnorm}{\right\|}
\nc{\rnormp}{\right\|_{\ell^{p,\eps}}}
\nc{\rar}{\rightarrow} 
\begin{document}
	
	\begin{abstract}
		
		We consider the higher order Schr\"odinger operator $H=(-\Delta)^m+V(x)$ in $n$ dimensions with real-valued potential $V$ when $n>2m$, $m\in \mathbb N$.    We adapt our recent results for $m>1$  to show that the wave operators are bounded on $L^p(\R^n)$ for the full the range $1\leq p\leq \infty$ in both even and odd dimensions without assuming the potential is small.  The approach used works without distinguishing even and odd cases,  captures the endpoints $p=1,\infty$, and  somehow simplifies the low energy argument even in the classical case of $m=1$.		 
	\end{abstract}

	\title[Wave operators for higher order  Schr\"odinger operators]{\textit{A note on endpoint $L^p$-continuity of wave operators for classical and higher order Schr\"odinger operators } } 
	
	\author[M.~B. Erdo\smash{\u{g}}an, W.~R. Green]{M. Burak Erdo\smash{\u{g}}an and William~R. Green}
	\thanks{ The first author was partially supported by the NSF grant  DMS-2154031 and Simons Foundation Grant 634269. The second author is partially supported by Simons Foundation
Grant 511825. }
	\address{Department of Mathematics \\
		University of Illinois \\
		Urbana, IL 61801, U.S.A.}
	\email{berdogan@illinois.edu}
	\address{Department of Mathematics\\
		Rose-Hulman Institute of Technology \\
		Terre Haute, IN 47803, U.S.A.}
	\email{green@rose-hulman.edu}

	\maketitle

\section{Introduction}

We consider equations of the form 
\begin{align*}
	i\psi_t =(-\Delta)^m\psi +V\psi, \qquad x\in \R^n, \quad  m\in \mathbb N.
\end{align*}
When $m=1$ this is the classical Schr\"odinger equation. Here $V$ is a real-valued potential with   polynomial decay, $|V(x)|\les \la x\ra^{-\beta}$ for some sufficiently large $\beta>0$.  
We denote the free operator by $H_0=(-\Delta)^m$ and the perturbed operator by $H=(-\Delta)^m+V$. We study the $L^p$ boundedness of the wave operators, which are defined by
$$
W_{\pm}=s\text{\ --}\lim_{t\to \pm \infty} e^{itH}e^{-itH_0}.
$$
The wave operators are of interest in scattering theory. For the classes of potentials $V$ we consider, the wave operators exist and are asymptotically complete,  \cite{RS,ScheArb,agmon,Hor,Sche}.   In addition,  we have the intertwining identity
$$
	f(H)P_{ac}(H)=W_\pm f((-\Delta)^m)W_{\pm}^*.
$$
Here $P_{ac}(H)$ is the projection onto the absolutely continuous spectral subspace of $H$, and $f$ is any Borel function.  This allows one to deduce $L^p$-based mapping properties of operators of the form $f(H)P_{ac}(H)$ from those of the much simpler operators $f((-\Delta)^m)$.  Other foundational work was done in \cite{Kur1,Kur2} in the context of scattering theory.

The usual starting point to study the $L^p$ boundedness of the wave operators is the stationary representation 
\begin{align}\label{eqn:stat rep}
	W_+u
	&=u-\frac{1}{2\pi i} \int_{0}^\infty \mR_V^+(\lambda) V [\mR_0^+(\lambda)-\mR_0^-(\lambda)] u \, d\lambda,
\end{align}
where $\mR_V(\lambda)=((-\Delta)^{m}+V-\lambda)^{-1}$, $\mR_0(\lambda)=((-\Delta)^m-\lambda)^{-1}$, and the `+' and `-' denote the usual limiting values as $\lambda$ approaches the positive real line from above and below, \cite{agmon,soffernew}.  It suffices to consider $W_+$ as $W_-=\mathcal C W_+ \mathcal C $, where $\mathcal Cu(x)=\overline{u}(x)$ is the conjugation operator.  Noting that the identity operator is bounded for all $1\leq p\leq \infty$, one needs only control the contribution of the integral involving the resolvent operators.

Using resolvent identity, we write
$$
\mR_V^+=\sum_{j=0}^{2k-1} (-1)^j \mR_0^+ (V\mR_0^+)^j +(\mR_0^+V)^{k} \mR_V^+ (V\mR_0^+)^k.
$$ 
We denote the contribution of the $j$th term of the finite sum to \eqref{eqn:stat rep} by $W_j$ and the contribution of the remainder by $W_{r,k}$.  To study the $L^p$ boundedness of   $W_{r,k}$ we need to consider whether $\lambda$ is in a neighborhood of zero or not.   To that end, let $\chi\in C^\infty_0$ be  a smooth cut-off function for a sufficiently small neighborhood of zero, with $\widetilde \chi=1-\chi$  the complementary cut-off away from zero. We define
$$
W_{low,k}u= \frac{1}{2\pi i}  \int_{0}^\infty \chi(\lambda) (\mR_0^+(\lambda) V)^{k} \mR_V^+(\lambda)  (V\mR_0^+(\lambda) )^k V [\mR_0^+(\lambda)-\mR_0^-(\lambda)] u \, d\lambda,
$$
$$
W_{high,k}u= \frac{1}{2\pi i}  \int_{0}^\infty \widetilde\chi(\lambda) (\mR_0^+(\lambda) V)^{k} \mR_V^+(\lambda)  (V\mR_0^+(\lambda) )^k V [\mR_0^+(\lambda)-\mR_0^-(\lambda)] u \, d\lambda.
$$
Throughout the paper, we write $\la x\ra$ to denote $  (1+|x|^2)^{\f12}$, $A\les B$ to say that there exists a constant $C$ with $A\leq C B$, and write $a-:=a-\epsilon$ and $a+:=a+\epsilon$ for some $\epsilon>0$.
Our main result is
\begin{theorem}\label{thm:main low}
	Let $n>2m\geq 2$. 
	Assume that $|V(x)|\les \la x\ra^{-\beta}$, where $V$ is a real-valued potential on $\R^n$ and $\beta>n+4$ when $n$ is odd and $\beta>n+3$ when $n$ is even.  Also assume $H=(-\Delta)^m+V(x)$ has no positive
	eigenvalues and zero energy is regular. Then $W_{low,k}$ extends to a  bounded operator on $L^p(\R^n)$ for all $1\leq p\leq \infty$ provided that $k$ is sufficiently large.
\end{theorem}
In fact the proof we supply works for all $k$ if $2m<n<4m$. We need sufficiently large $k$ when $n\geq 4m$ due to local singularities of the free resolvents. 
We note that this result is new only in the endpoint cases $p=1,\infty$ when $m>1$ and $n>2m$ even.  The main novelty is that our method applies to all cases $n>2m, m\geq 1, 1\leq p\leq \infty$ in one self-contained argument,  see  Proposition~\ref{lem:low tail low d} below.

To put this result in the context we recall that  the first $L^p$ boundedness result is   the seminal paper of Yajima, \cite{YajWkp1}, for $m=1$.  By controlling the Born series terms, the result was shown to hold for all $1\leq p\leq \infty$ for small potentials.  To remove this smallness assumption, the main difficulty is in controlling the contribution of $W_{low, k}$.  The behavior of this operator differs in even and odd dimensions. In \cite{YajWkp1,YajWkp2,YajWkp3}, Yajima provided arguments that removed smallness or positivity assumptions on the potential for all dimensions $n\geq 3$.  Yajima later simplified these arguments and considered the effect of zero energy eigenvalues and/or resonances in \cite{Yaj} when $n$ is odd and with Finco in \cite{FY} when $n$ is even for $n>4$.

We now give more details in the case $m>1$ to state the new corollary of our result above on the $L^p$ boundedness of wave operators.  
Let $H^\delta$ be the Sobolev space of functions with $\|\la \cdot \ra^{\delta} \mF({f})\|_2<\infty$, where $\mF({f})$ denotes the Fourier transform of $f$.
\begin{asmp}\label{asmp:Fourier}
	For some $0<\delta \ll 1$, assume that the real-valued potential $V$ satisfies the condition
	\begin{enumerate}[i)]
		 
		\item $\big\| \la \cdot \ra^{\frac{4m+1-n}{2}+\delta} V(\cdot)\big\|_{2}<C$   when $2m<n<4m-1$,
		
		\item $\big\|\la \cdot \ra^{1+\delta}V(\cdot)\big\|_{H^{\delta}}<C$ when $n=4m-1$,
		
		\item  $\big\|\mathcal F(\la \cdot \ra^{\sigma} V(\cdot))\big\|_{L^{ \frac{n-1-\delta}{n-2m-\delta} }}<C$ for some $\sigma>\frac{2n-4m}{n-1-\delta}+\delta$   when $n>4m-1$.
	\end{enumerate}
	
\end{asmp}
In \cite{EGWaveOp}, by adapting Yajima's $m=1$ argument in \cite{YajWkp1}, it was shown that the contribution of the terms of the Born series may be bounded by
\begin{align*}
	\|W_j\|_{L^p\to L^p}\leq C^j  \| V\|_{n,m}^j, 
\end{align*}	
where $\|V\|_{n,m}$ denotes the norm used in Assumption~\ref{asmp:Fourier} when $m>1$ for the different ranges of $n$ considered.
In addition, it was shown that if $|V(x)|\les \la x\ra^{-\beta}$ for some $\beta>n+5$ when $n$ is odd and $\beta>n+4$ when $n$ is even and if $k$ is sufficiently large (depending on $m$ and $n$), then $W_{high,k}$ is a bounded operator on $L^p$ for all $1\leq p\leq \infty$.

Combining these facts with Theorem~\ref{thm:main low}, we have the following   result which is new in the case $n$ is even.

\begin{corollary}\label{thm:full} 
	Fix  $m>1$ and let $n>2m$. 
	Assume that  $V$  satisfies Assumption~\ref{asmp:Fourier} and in addition
	\begin{enumerate}[i)]
		\item $|V(x)|\les \la x\ra^{-\beta}$  for some $\beta>n+5$ when  $n$ is odd and 
		for some  $\beta>n+4$  when $n$ is even,
		\item $H=(-\Delta)^m+V(x)$ has no positive
		eigenvalues and zero energy is regular. 
	\end{enumerate}
	Then,	the wave operators extend to bounded operators on $L^p(\R^n)$ for all $1\leq p\leq\infty$.  
	
\end{corollary}
By applying the intertwining identity and the known $L^1\to L^\infty$ dispersive bound of the free solution operator $e^{-it(-\Delta)^m}$ and for $(-\Delta)^{\frac{n(m-1)}{2}}e^{-it(-\Delta)^m}$, we obtain the corollary below. 
The second bound below was  observed in \cite{EGG} for the free operator, and was used  to obtain counterexamples for the $L^p$ boundedness of wave operators. 
\begin{corollary}
	Under the assumptions of Corollary~\ref{thm:full}, we obtain the global dispersive estimates
	$$
		\|e^{-itH}P_{ac}(H)f\|_{\infty} \les |t|^{-\frac{n}{2m}} \|f\|_1.
	$$
	$$
		\|H^{\frac{n(m-1)}{2m}}e^{-itH}P_{ac}(H)f\|_{\infty} \les |t|^{-\frac{n}{2}} \|f\|_1.
	$$
\end{corollary}

The $L^p$ boundedness of the wave operators in the higher order $m>1$ case has only recently been studied.  The first result was in the case of $m=2$ and $n=3$ by Goldberg and the second author, \cite{GG4wave}.  Here the wave operators were shown to be bounded on $1<p<\infty$. More recently in \cite{EGWaveOp},  we  proved $1\leq p\leq \infty$ boundedness in the cases when $n>2m$ and $m>1$  for small potentials or $n$ odd for large potentials.  In \cite{MWY},  Mizutani, Wan, and Yao  considered the case of $m=4$ and $n=1$ showing that the wave operators are bounded when $1<p<\infty$, but not when $p=1,\infty$, where weaker estimates involving the Hardy space or BMO were proven.  This recent work on higher order, $m>1$, Schr\"odinger operators has roots in the work of Feng, Soffer, Wu and Yao \cite{soffernew} which considered time decay estimates between weighted $L^2$ spaces.

The paper is organized as follows.  In Section~\ref{sec:low} we collect facts about resolvent operators needed to prove Theorem~\ref{thm:main low}.  In Section~\ref{sec:prop pf}  we establish the main technical tool, Proposition~\ref{lem:low tail low d}.  In Section~\ref{sec:Minv} we prove several technical lemmas which, in particular, show that Proposition~\ref{lem:low tail low d} implies Theorem~\ref{thm:main low}.

\section{Resolvent Expansions}\label{sec:low}

In this section we lay the groundwork to prove the low energy result, Theorem~\ref{thm:main low}.    It is convenient to use a change of variables to respresent $W_{low,k}$ as  
$$
	\frac{m}{\pi i}\int_{0}^\infty \chi(\lambda) \lambda^{2m-1} (\mR_0^+(\lambda^{2m}) V)^{k } \mR_V^+(\lambda^{2m})  (V\mR_0^+(\lambda^{2m}) )^k V   [\mR_0^+(\lambda^{2m})-\mR_0^-(\lambda^{2m})]  \, d\lambda
$$
We begin by using the symmetric resolvent identity on the perturbed resolvent $\mR_V^+(\lambda^{2m})$. With $v=|V|^{\f12}$, $U(x)=1$ if $V(x)\geq 0$ and $U(x)=-1$ if $V(x)<0$, we define $M^+(\lambda)=U+v\mR_0^+(\lambda^{2m})v$. Recall that $M^+$ is invertible on $L^2$   in a sufficiently small neighborhood of $\lambda=0$ provided that zero is a regular point of the spectrum. 
Using the symmetric resolvent identity, one has
$$
\mR_V^+(\lambda^{2m})V=\mR_0^+(\lambda^{2m})vM^+(\lambda)^{-1}v.
$$
We select the cut-off $\chi$ to be supported in this neighborhood.  Therefore, we have
$$
W_{low,k}=	\frac{m}{\pi i}\int_{0}^\infty \chi(\lambda) \lambda^{2m-1} \mR_0^+(\lambda^{2m}) v\Gamma_k(\lambda) v   [\mR_0^+(\lambda^{2m})-\mR_0^-(\lambda^{2m})]  \, d\lambda,
$$
where $\Gamma_0(\lambda):= M^+(\lambda)^{-1} $ and for $k\geq 1$
\be\label{Gammalambda}
\Gamma_k(\lambda):= Uv\mR_0^+(\lambda^{2m}) \big(V\mR_0^+(\lambda^{2m})\big)^{k-1} vM^+(\lambda)^{-1}v \big( \mR_0^+(\lambda^{2m})V\big)^{k-1} \mR_0^+(\lambda^{2m})vU.
\ee
To state the main result of this section, we define an operator $T:L^2\to L^2$ with integral kernel $T(x,y)$ to be absolutely bounded if the operator with kernel $|T(x,y)|$ is bounded on $L^2$. 
\begin{prop}\label{lem:low tail low d} Fix  $n>2m\geq 2$  and let $\Gamma$ be a $\lambda $ dependent absolutely bounded operator. Let 
$$
\widetilde \Gamma (x,y):= \sup_{0<\lambda <\lambda_0}\Big[|\Gamma(\lambda)(x,y)|+ \sup_{1\leq k\leq  \lceil \tfrac{n}2\rceil +1  } \big|\lambda^{k-1} \partial_\lambda^k \Gamma(\lambda)(x,y)\big| \Big].
$$
For $2m<n<4m$ assume that $\widetilde \Gamma $ is bounded on $L^2$, and for $n\geq 4m$  assume that $\widetilde\Gamma $ satisfies   \be\label{eq:tildegamma}
\widetilde \Gamma (x,y)  \les \la x\ra^{-\frac{n}2-}\la y\ra^{-\frac{n}2-}.
\ee Then the operator with kernel 
	\be\label{Kdef}
	K(x,y)=\int_0^\infty \chi(\lambda) \lambda^{2m-1} \big[\mR_0^+(\lambda^{2m}) v \Gamma(\lambda)  v [\mR_0^+(\lambda^{2m}) -\mR_0^-(\lambda^{2m})]\big](x,y)   d\lambda 
	\ee
	is bounded on $L^p$ for $1\leq p\leq \infty$ provided that $\beta>n$.
\end{prop}
Note that  Theorem~\ref{thm:main low} follows from this proposition and the following 
\begin{lemma}\label{lem:Gamma}
 Fix  $n>2m\geq 2$. Assume that $|V(x)|\les \la x\ra^{-\beta}$, where  $\beta>n+4$ when $n$ is odd and $\beta>n+3$ when $n$ is even.  Also assume that zero is a regular point of the spectrum of $H$.
 Then the operator $\Gamma_k(\lambda) $ defined in \eqref{Gammalambda} satisfies the hypothesis of Proposition~\ref{lem:low tail low d} for all $k$ when $2m<n<4m$ and for all sufficiently large $k$ when $n\geq 4m$.
\end{lemma}
We prove Propostion~\ref{lem:low tail low d} in Section~\ref{sec:prop pf}, and provide the argument for Lemma~\ref{lem:Gamma} in Section~\ref{sec:Minv}.  
To prove these results  we need the following representations of  the free resolvent given in Lemmas~3.2 and 6.2 in \cite{EGWaveOp}.

\begin{lemma}\label{prop:F} Let $n>2m\geq 2$. Then,   we have  
	$$
	\mR_0^+(\lambda^{2m})(y,u)
	=  \frac{e^{i\lambda|y-u|}}{|y-u|^{n-2m}} F(\lambda |y-u| ).$$
When $r\gtrsim 1$, we have  $|F^{(N)}(r)|\les  r^{\frac{n+1}2 -2m-N}$ for all $N$. When $r\ll 1$ and $n$ is odd, we have $|F^{(N)}(r)|\les 1 $ for all $N$. When $r\ll 1$ and $n$ is even, we have 
	$$ |F^{(N)}(r)|\les \left\{ \begin{array}{ll} 1&N=0,1,\ldots,2m-1,\\ |\log(r)|& N=2m,\\ r^{2m-N},   & N>2m.\end{array}\right. 
	$$  
\end{lemma}

These estimates won't suffice for our purposes; we also need to take advantage of cancellation in the difference $\mR_0^+(\lambda^{2m})-\mR_0^-(\lambda^{2m})$. 

\begin{lemma}\label{lem:R0cancel}
Let $n>2m\geq 2$. We have 
$$
[\mR_0^+(\lambda^{2m})-\mR_0^-(\lambda^{2m})](y,u)= \lambda^{n-2m}  \big[ e^{i\lambda |y-u|}F_+(\lambda |y-u|)+e^{-i\lambda |y-u|}F_-(  \lambda |y-u|)\big],
$$
where $  F_{\pm}$ are $C^\infty$ functions on $\R$ satisfying  for all $j\geq 0$, $r\in\R$ 
$$
|\partial_r^j F_\pm(r)|\les \la r\ra^{\frac{1-n}2-j}.
$$
\end{lemma} 
\begin{proof} By the splitting identity, we have
\be\label{splitdiff}[\mR_0^+-\mR_0^-](\lambda^{2m})(y,u)=\frac{1}{m\lambda^{2m-2}}[R_0^+-R_0^-](\lambda^2)(y,u).
\ee
Since, $[R_0^+-R_0^-](\lambda^2)$ is a multiple of the imaginary part of $R_0^+$.  Since this may be expressed as a multiple of $(\frac{\lambda}{|y-u|})^{\frac{n-2}{2}}J_{\frac{n-2}{2}}(\lambda|y-u|)$ with $J_{\frac{n-2}{2}}$ a Bessel function of the first kind, we have  
$$
 [\mR_0^+-\mR_0^-](\lambda^{2m})(y,u)=C_{m,n} \lambda^{2-2m} \lambda^{n-2}\sum_{j=0}^\infty c_{n,j} (\lambda |y-u|)^{2j}=:\lambda^{n-2m}  \widetilde F(\lambda|y-u|),
$$
where $|c_j|\les \tfrac1{j!}$.	This proves that $\widetilde F$ is entire and with bounded derivatives for $|r|\les 1$.  Since $\cos(r) \geq \frac12$ for $|r|\ll 1$ we can write 
$$
\widetilde F(r) \chi(r)= e^{ir} \frac{\chi(r) \widetilde F(r)}{2\cos(r)} + e^{-ir}  \frac{\chi(r) \widetilde F(r)}{2\cos(r)}. 
$$

For $|r|\gtrsim 1$, using the representation (ignoring constants)
$$
[R_0^+-R_0^-](\lambda^2)(y,u)=\Big( \frac{\lambda}{|y-u|}\Big)^{\frac{n}2-1}\big(e^{i\lambda |y-u|} \omega_+(\lambda |y-u|) + e^{-i\lambda |y-u|} \omega_-(\lambda |y-u|)\big),
$$
where $\omega_\pm(r)= \widetilde O(|r|^{-1/2})$, we see that
$$
\widetilde \chi(r) \widetilde F(r)= e^{ir} \widetilde \chi(r) r^{1-\frac{n}2}  \omega_+(r) + e^{-ir} \widetilde \chi(r) r^{1-\frac{n}2}  \omega_-(r)).
$$
This yields the bounds for $|r|\gtrsim 1$ after identifying 
$$
F_\pm(r)= \frac{\chi(r) \widetilde F(r)}{2\cos(r)} +  \widetilde \chi(r) r^{1-\frac{n}2}  \omega_\pm(r).
$$
\end{proof}

\begin{rmk}
The effect of $\lambda$ derivatives on $F(\lambda r)$ and $ F_\pm(\lambda r) $ can be bounded by division by $\lambda$, i.e., for all $N=0,1,2,...,$ and for all $n>2m\geq 2$, we have 
\be\label{Fbounds}
|\partial_\lambda^N F(\lambda r)|\les \lambda^{-N} \la  \lambda r \ra^{\frac{n+1}2 -2m},\, \,\,\, \,\,\, 
|\partial_\lambda^N F_\pm(\lambda r)|\les \lambda^{-N} \la  \lambda r\ra^{\frac{1-n}2}.
\ee
This  is clear for $F_\pm$ and also for $F$ except when $n$ is even, $N\geq 2m$ and $\lambda r\ll 1$, in which case the bound also holds since 
$$
r^{N}|F^{(N)}(\lambda r)|\les r^{N} (\lambda r)^{2m-N-}  \les r^{N} (r\lambda)^{-N}=\lambda^{-N}. 
$$
\end{rmk}
Another corollary of Lemma~\ref{prop:F} is  
\begin{corollary}\label{cor:E} Let $E(\lambda ) 
(r):=\mR_0^+(\lambda^{2m})(r)-\mR_0^+(0)(r)$. Then, for $\lambda r\ll 1$,  we have 
$$
| \partial_\lambda^N E(\lambda)(r)|\les \lambda^{1-N} r^{2m-n+1}, \,\,\,\,N=0,1,2,...
$$
When $\lambda r\gtrsim 1$, we have 
$$| E(\lambda)(r)| \les r^{\frac{1-n}{2} }\lambda^{\frac{n+1}2-2m} +r^{2m-n}, \text{ and}$$
$$|\partial_\lambda^{N}E(\lambda)(r)| \les r^{\frac{1-n}{2}+N}\lambda ^{\frac{n+1}2-2m}, \,\,\,\,N=1,2,...$$  
\end{corollary}
\begin{proof} First consider the case  $\lambda r\ll 1$.
For $N=0$, the claim follows from the mean value theorem and Lemma~\ref{prop:F}. For $N\geq 1$, again by Lemma~\ref{prop:F}, we have for odd $n$
$$
\lambda^N |\partial_\lambda^N E(\lambda)|\les \lambda^N r^N r^{2m-n} \les  \lambda r^{2m-n+1}.
$$
The proof for n even is similar, using an adjustment for $N\geq 2m$ as in the remark above.

When $\lambda r\gtrsim 1$, the worst case is when the derivatives hit the exponential, which gives the inequality when $N\geq 1$. When $N=0$ the two summands correspond  to the contributions of 
$\mR_0^+(\lambda^{2m})(r)$ and $\mR_0^+(0)(r)$.
\end{proof}

\section{Proof of Proposition~\ref{lem:low tail low d}}\label{sec:prop pf} 

We say an operator $K$ with integral kernel $K(x,y)$ is admissible if
$$
\sup_{x\in \R^n} \int_{\R^n} |K(x,y)|\, dy+	\sup_{y\in \R^n} \int_{\R^n} |K(x,y)|\, dx<\infty.
$$
By the Schur test, it follows that an operator with admissible kernel is bounded on $L^p(\R^n)$ for all $1\leq p\leq \infty$. We are now ready to prove Proposition~\ref{lem:low tail low d}.

\begin{proof}[Proof of Proposition~\ref{lem:low tail low d}]
	Using the representations in Lemma~\ref{prop:F} and Lemma~\ref{lem:R0cancel} with $r_1=|x-z_1|$ and $r_2:=|z_2-y|$ we see that $K(x,y)$ is the difference of  
 	\be \label{Kdefini}
 	K_\pm(x,y)=\int_{\R^{2n}}  \frac{ v(z_1) v(z_2) }{r_1^{n-2m}   } \int_0^\infty e^{i\lambda (r_1 \pm r_2 )} \chi(\lambda) \lambda^{n-1} \Gamma(\lambda)(z_1,z_2) F(\lambda r_1)F_\pm(\lambda r_2)  d\lambda dz_1 dz_2 .
 	\ee
 	We write
 	$$
 	K(x,y) =: \sum_{j=1}^4 K_{j}(x,y),
 	$$
 	where the integrand in  
 	$K_1$ is  restricted  to the set $r_1,r_2\les 1$, in $K_2$ to the set $r_1 \approx r_2\gg 1 $,
 	in $K_3$ to the set $r_2\gg \la r_1\ra $,  in $K_4$ to the set $r_1 \gg \la r_2\ra$. We define $K_{j,\pm}$ analogously. 
 	
 	Using the bounds of Lemmas~\ref{prop:F} and \ref{lem:R0cancel} for $\lambda r\ll1$,  we bound the contribution of 	$|K_{1,\pm}(x,y)|$ by 
 	$$
 	\int_{\R^{2n}}  \frac{ v(z_1) v(z_2)\chi_{r_1,  r_2\les 1} }{r_1^{n-2m}   }   \widetilde\Gamma (z_1,z_2)   dz_1 dz_2.
 	$$
 	Therefore
 	$$
 	\int  |K_{1,\pm}(x,y)|dx    \les \big\|| \cdot|^{2m-n}\big\|_{L^1(B(0,1))} \|v\|_{L^2}^2 \|\widetilde \Gamma\|_{L^2\to L^2}\les 1,
 	$$
 	uniformly in $y$. Similarly, provided that $2m<n<4m$,
 	$$\int  |K_{1,\pm}(x,y)|dy\les \|\widetilde\Gamma\|_{L^2\to L^2} \|v\|_{L^2} \|v(\cdot) |x-\cdot|^{2m-n}\|_{L^2}\les 1 
 	$$
 	holds uniformly in $x$. When $n\geq 4m$, we use the decay bound \eqref{eq:tildegamma} on $\widetilde \Gamma$ to obtain
 	$$\int  |K_{1,\pm}(x,y)|dy\les \int \la z_1\ra^{-n-}\la z_2 \ra^{-n-} r_1^{2m-n} dz_1dz_2 \les 1,  
 	$$
 	which implies that $K_1$ is admissible. 
 	
 	For $K_2$, we restrict ourself to $K_{2,-}$ since the $+$ sign is easier to handle. We integrate by parts twice in the $\lambda$ integral when $\lambda |r_1-r_2|\gtrsim1$ (using \eqref{Fbounds} and the definition of $\widetilde \Gamma$) and estimate directly when $\lambda |r_1-r_2|\ll1$ to obtain 
 	\begin{align} 
 		\nonumber
 		|K_{2,-}(x,y)|\les \int_{\R^{2n}} \frac{v(z_1)  \widetilde\Gamma (z_1,z_2) v(z_2)\chi_{r_1\approx r_2\gg 1 }}{  r_1^{n-2m}  } \int_0^\infty  \chi(\lambda) \lambda^{n-1} \chi(\lambda|r_1-r_2|)  \la\lambda r_1\ra^{1-2m}    d\lambda dz_1 dz_2 \\ \nonumber
 		+ \int_{\R^{2n}} \frac{v(z_1)  \widetilde\Gamma (z_1,z_2) v(z_2)\chi_{r_1\approx r_2\gg 1 }}{  r_1^{n-2m}  } \int_0^\infty  \frac {\chi(\lambda) \lambda^{n-3} \widetilde\chi(\lambda|r_1-r_2|)   \la\lambda r_1\ra^{1-2m}}{|r_1-r_2|^2}     d\lambda dz_1 dz_2\\
 		\les  \int_{\R^{2n}} \frac{v(z_1)  \widetilde\Gamma (z_1,z_2) v(z_2)\chi_{r_1\approx r_2\gg 1 }}{  r_1^{n-2m}  } \int_0^\infty  \frac {\chi(\lambda) \lambda^{n-1}   \la\lambda r_1\ra^{1-2m}}{\la \lambda (r_1-r_2)\ra^2}     d\lambda dz_1 dz_2. \nonumber
 	\end{align}
 	Therefore, passing to the polar coordinates in $x$ integral (centered at $z_1$) and noting 
 	$1-2m<0$, we have 
 	$$
 	\int|K_{2,-}(x,y)| dx 
 	\les  \int_{\R^{2n}} \int_0^1 \int_{r_1\approx r_2\gg 1}  v(z_1)  \widetilde\Gamma (z_1,z_2) v(z_2)    \frac {  \lambda^{n-2m}   }{\la \lambda (r_1-r_2)\ra^2}      d r_1 d\lambda dz_1 dz_2
 	$$
 	$$  \les  \int_{\R^{2n}} \int_0^1 \int_{\R}  v(z_1)  \widetilde\Gamma (z_1,z_2) v(z_2)    \frac {  \lambda^{n-2m-1}   }{\la  \eta \ra^2}      d \eta  d\lambda dz_1 dz_2\les 1, 
 	$$
 	uniformly in $y$. In the second line we defined $\eta=\lambda (r_1-r_2)$ in the $r_1$ integral and used   $n-2m-1\geq 0$.
 	Since $r_1\approx r_2$, the integral in $y$ can be bounded uniformly in $x$ and hence the contribution of $K_2$ is admissible.  We now consider the contribution of
 	\begin{multline} \label{K4ndef}
 		K_{4,\pm}(x,y)= 
 		\int_{\R^{2n}}  \frac{v(z_1) v(z_2)\chi_{r_1 \gg \la r_2\ra}}{r_1^{n-2m}}  \\  \int_0^\infty  e^{i  \lambda (r_1\pm r_2) } F(  \lambda r_1) \chi(\lambda)   \Gamma(\lambda)(z_1,z_2) \lambda^{n-1}  F_\pm(\lambda r_2)  \ d\lambda dz_1 dz_2 .
 	\end{multline}
 	When $\lambda r_1\les 1$, using \eqref{Fbounds}, we bound $|F_\pm(\lambda r_2)|, |F(\lambda r_1)|\les 1$ and estimate the $\lambda$ integral  by 
 	$r_1^{-n} \widetilde \Gamma(z_1,z_2)$, whose contribution to  $K_4$ is bounded by 
 	$$
 	\int_{\R^{2n}}  \frac{v(z_1) v(z_2) \widetilde \Gamma(z_1,z_2) \chi_{r_1 \gg \la r_2\ra}}{r_1^{n+1}} dz_1 dz_2 .
 	$$
 	Which, by Lemma~\ref{lem:admissible}, is admissible. 
 	
 	When $\lambda r_1 \gtrsim 1$, we integrate by parts $N=\lceil n/2\rceil +1 $ times 
 	(using \eqref{Fbounds}) to obtain the bound
 	$$
 	\frac1{|r_1\pm r_2|^{N}}\int_0^\infty \Big|\partial_\lambda^{N} 
 	\big[F(  \lambda r_1) \widetilde\chi(\lambda r_1)  \chi(\lambda)  \lambda^{n-1} \Gamma(\lambda)(z_1,z_2)   F_\pm(\lambda r_2) \big] \Big| d\lambda 
 	$$  
 	$$
 	\les r_1^{-N} \sum_{0\leq j_1+j_2+j_3+j_4\leq N, \, j_i\geq 0  }  
 	\int_{\frac1{r_1}}^1    \lambda^{\frac{n+1}2-2m-j_1}  r_1^{\frac{n+1}2-2m}   \lambda^{n-1-j_2}  \big|\partial_\lambda^{j_3}\Gamma(\lambda)(z_1,z_2)  \big|\frac{\lambda^{-j_4}}{\la \lambda r_2\ra^{\frac{n-1}2}  } d\lambda 
 	$$
 	$$
 	\les r_1^{ \frac{n+1}2-2m-N }  \widetilde \Gamma(z_1,z_2) \sum_{0\leq j_1+j_2+j_3+j_4\leq N, \, j_i\geq 0  }  
 	\int_{\frac1{r_1}}^1    \lambda^{\frac{3n-1}2-2m-j_1-j_2-j_3-j_4}        d\lambda 
 	$$
 	$$
 	\les r_1^{ \frac{n+1}2-2m-N }  \widetilde \Gamma(z_1,z_2) \int_{\frac1{r_1}}^1    \lambda^{\frac{3n-1}2-2m-N}  d\lambda  \les  r_1^{ \frac{n+1}2-2m-N }  \log(r_1)  \widetilde \Gamma(z_1,z_2).
 	$$
 	In the last inequality we noted that 
 	$\frac{3n-1}2-2m-N\geq -1$, so the $\lambda$ integral is either bounded or grows like $\log (r_1)\les r_1^{0-}$ for $r_1\ll 1$.  Noting that $\lceil \frac{n}{2}\rceil +\frac{n+1}{2}\geq n+\frac{1}{2}$, the contribution of this to \eqref{K4ndef} is 
 	$$
 	\les \int_{\R^{2n}}  \frac{v(z_1) v(z_2)\chi_{r_1 \gg \la r_2\ra}}{r_1^{n+\frac12-}} dz_1 dz_2.
 	$$
 	By Lemma~\ref{lem:admissible}, this  is admissible.

 	We now consider $K_3$, which is the most challenging case. 
 	Using \eqref{splitdiff} in \eqref{Kdef}, we write
 	\begin{multline} \label{K3ndef1}
 		K_3(x,y)=  
 		\int_{\R^{2n}}  \chi_{r_2 \gg \la r_1\ra} v(z_1) v(z_2)\\ 	\int_0^\infty \lambda \chi(\lambda)   \mR_0^+(\lambda^{2m}) (r_1)    \Gamma(\lambda)(z_1,z_2) [R_0^+(\lambda^2)- R_0^-(\lambda^2)](r_2)  \ d\lambda dz_1 dz_2 .
 	\end{multline}
 	We write
 	$$
 	\mR_0^+(\lambda^{2m})=\mR_0^+(0) + [\mR_0^+(\lambda^{2m})-\mR_0^+(0)]=: G_0+E(\lambda),
 	$$
 	$$
 	\Gamma(\lambda)= \Gamma(0)+[\Gamma(\lambda)-\Gamma(0)]=:\Gamma(0)+\Gamma_1(\lambda).
 	$$
 	Here $G_0=\mR_0^+(0)=c_{n,m}r_1^{2m-n}$.
 	By considering $\mR_0^+(\lambda^{2m})\Gamma(\lambda)$ as a perturbation of $\mR_0^{+}(0)\Gamma(0)$, we can show the kernel is admissible and capture the endpoint, $p=1,\infty$, boundedness.
 	We first consider the contribution of $G_0\Gamma(0)$ to $K_3$:
 	$$
 	\int_{\R^{2n}}  \chi_{r_2 \gg \la r_1\ra} v(z_1) v(z_2) 	G_0(r_1)\Gamma(0)(z_1,z_2) \int_0^\infty \lambda \chi(\lambda)    [R_0^+(\lambda^2)- R_0^-(\lambda^2)](r_2)  \ d\lambda dz_1 dz_2. 
 	$$
 	Identifying the $\lambda$ integral as a constant multiple of the kernel of $\chi(\sqrt{-\Delta})$,   we may bound it as $O(\la r_2\ra^{-N})$ for all $N$   since $\chi(|\xi|)$ is Schwartz.  Therefore, we have the bound
 	$$
 	\int_{\R^{2n}}  \chi_{r_2 \gg \la r_1\ra} v(z_1) v(z_2) 	r_1^{2m-n} r_2^{-n-1} \widetilde \Gamma (z_1,z_2) dz_1 dz_2,
 	$$
 	which is admissible by Lemma~\ref{lem:admissible}. 
 	
 	It remains to consider the contributions of $\mR^+_0(\lambda^{2m})\Gamma_1(\lambda) $ and of $ E(\lambda) \Gamma(0)$.
 	The former can be written as 
 	$$
 	\int_{\R^{2n}}  \frac{v(z_1) v(z_2)\chi_{r_2 \gg \la r_1\ra}}{r_1^{n-2m}}  \int_0^\infty  e^{i  \lambda (r_1\pm r_2) } F(  \lambda r_1) \chi(\lambda)  \lambda^{n-1}   \Gamma_1(\lambda)(z_1,z_2)  F_\pm(\lambda r_2)  \ d\lambda dz_1 dz_2 .
 	$$
 	When $\lambda r_2\ll1$, using $|\Gamma_1(\lambda)|\les \lambda \widetilde \Gamma$, which follows from the mean value theorem, and \eqref{Fbounds} to directly integrating in $\lambda$, we obtain the bound
 	$$
 	\int_{\R^{2n}}  \frac{v(z_1) v(z_2)\chi_{r_2 \gg \la r_1\ra}}{r_1^{n-2m}r_2^{n+1}}\widetilde \Gamma(z_1,z_2)    dz_1 dz_2,
 	$$
 	which is admissible by Lemma~\ref{lem:admissible}. When $\lambda r_2 \gtrsim 1$, integrating by parts $N=\lceil n/2\rceil +1$ times, we have the bound
 	\be\label{K3former}
 	\int_{\R^{2n}}  \frac{v(z_1) v(z_2)\chi_{r_2 \gg \la r_1\ra}}{r_1^{n-2m}|r_1\pm r_2|^N}  \int_0^\infty  \Big| \partial_\lambda^N \big[ F(  \lambda r_1) \chi(\lambda) \widetilde \chi(\lambda r_2)  \lambda^{n-1}   \Gamma_1(\lambda)(z_1,z_2)  F_\pm(\lambda r_2)\big]\Big|  \ d\lambda dz_1 dz_2 .
 	\ee
 	We estimate the $\lambda $ integral by (noting that $\lambda^{j_3} |\partial^{j_3}_\lambda \Gamma_1|\les \lambda \widetilde \Gamma$ and using \eqref{Fbounds})
 	$$
 	\les r_2^{-\frac{n-1}2} \widetilde\Gamma(z_1,z_2)  \sum_{0\leq j_1+j_2+j_3+j_4\leq N, \, j_i\geq 0  }  
 	\int_{\frac1{r_2}}^1  \la  \lambda r_1\ra^{\frac{n+1}2-2m} \lambda^{-j_1} \lambda^{n-1-j_2}      \lambda^{1-j_3}  \lambda^{-\frac{n-1}2-j_4} d\lambda 
 	$$
 	$$
 	\les r_2^{-\frac{n-1}2} \widetilde\Gamma(z_1,z_2)   
 	\int_{\frac1{r_2}}^1  \la  \lambda r_1\ra^{\frac{n+1}2-2m}   \lambda^{\frac{n+1}2-\lceil \frac{n}2\rceil -1} d\lambda 
 	$$
 	$$
 	\les r_2^{-\frac{n-1}2} \widetilde\Gamma(z_1,z_2) 
 	\Big(\int_{\frac1{r_2}}^{\min(\frac1{r_1},1)}  \lambda^{\frac{n }2-\lceil \frac{n}2\rceil -\frac12 }      d\lambda +  \int_{\min(\frac1{r_1},1)}^1 r_1^{\frac{n+1}2-2m}  \lambda^{\lfloor \frac{n}2\rfloor -2m }       d\lambda  \Big).
 	$$
 	Note that $n-\lceil \frac{n}{2}\rceil=\lfloor \frac{n}{2}\rfloor$ is used in the final integral.
 	The first integral is at most $\log(r_2)$. Since $\frac{n-1}{2}+N\geq n+\frac{1}{2}$, its contribution to \eqref{K3former} is at most
 	$$
 	\int_{\R^{2n}}  \frac{v(z_1) v(z_2)\chi_{r_2 \gg \la r_1\ra}}{r_1^{n-2m}r_2^{n+\frac12-}}\widetilde \Gamma(z_1,z_2)    dz_1 dz_2,
 	$$
 	which is admissible by Lemma~\ref{lem:admissible}. Similarly,  the second integral is bounded by  $r_1^{n-2m}$ after multiplying the integrand by $(\lambda r_1)^{\frac{n-1}2}$. Contribution of this to \eqref{K3former} is at most
 	$$
 	\int_{\R^{2n}}  \frac{v(z_1) v(z_2)\chi_{r_2 \gg \la r_1\ra}}{ r_2^{n+\frac12}}\widetilde \Gamma(z_1,z_2)    dz_1 dz_2,
 	$$
 	which, by Lemma~\ref{lem:admissible}, is also admissible.

 	We now consider the contribution of $E(\lambda)\Gamma(0)$:  
 	\be\label{egamma}
 	\int_{\R^{2n}}   v(z_1) v(z_2)\chi_{r_2 \gg \la r_1\ra}\Gamma(0)(z_1,z_2)  \int_0^\infty  e^{\pm  i  \lambda   r_2 }  E(\lambda)(r_1) \chi(\lambda)   \lambda^{n-1}  F_\pm(\lambda r_2)  \ d\lambda dz_1 dz_2 .
 	\ee
 	Using Lemma~\ref{prop:F} and Corollary~\ref{cor:E} when $\lambda r_1\ll 1$.  Using this 
 	when $\lambda r_2\ll 1$ and using $|\Gamma(0)(z_1,z_2)|\leq \widetilde\Gamma(z_1,z_2)$, we bound \eqref{egamma} by direct estimate by 
 	$$
 	\int_{\R^{2n}}  \frac{ v(z_1) v(z_2)\chi_{r_2 \gg \la r_1\ra}\widetilde\Gamma(z_1,z_2) }{r_1^{n-2m-1} r_2^{n+1} }  dz_1 dz_2,
 	$$  
 	which is admissible by Lemma~\ref{lem:admissible} since $n-2m\geq 1$. 
 	
 	When $\lambda r_2\gtrsim 1$ and $\lambda r_1 \ll 1$, we integrate by parts $N=\lceil n/2\rceil +1  $ times to obtain
 	$$
 	\int_{\R^{2n}}   r_2^{ -N}  v(z_1) v(z_2)\chi_{r_2 \gg \la r_1\ra}\widetilde \Gamma (z_1,z_2)   \int_0^\infty  \Big|\partial_\lambda^N\big[  E(\lambda)(r_1) \chi(\lambda)  \chi(\lambda r_1)  \widetilde  \chi(\lambda r_2)  \lambda^{n-1}  F_\pm(\lambda r_2)\big]\Big|  \ d\lambda dz_1 dz_2. 
 	$$
 	Using Corollary~\ref{cor:E} and \eqref{Fbounds}, we bound this by 
 	$$
 	\int_{\R^{2n}}  r_2^{-N} r_1^{1+2m-n} v(z_1) v(z_2)\chi_{r_2 \gg \la r_1\ra}\widetilde \Gamma (z_1,z_2)  \int_{r_2^{-1}}^1   \lambda^{n-\lceil n/2\rceil-1}  (\lambda r_2)^{\frac{1-n}2}  \ d\lambda dz_1 dz_2 
 	$$
 	$$
 	\les  \int_{\R^{2n}}  r_2^{-n-\frac12} r_1^{1+2m-n} v(z_1) v(z_2)\chi_{r_2 \gg \la r_1\ra}\widetilde \Gamma (z_1,z_2)  dz_1 dz_2,
 	$$
 	which is again admissible by Lemma~\ref{lem:admissible}.

 	It remains to consider the case $\lambda r_1\gtrsim 1$. Integrating by parts once we rewrite the $\lambda$ integral in \eqref{egamma} as  
 	\begin{align}\label{eq:K3 bad1}
 		&\frac1{r_2} \int_0^\infty  e^{\pm  i  \lambda   r_2 }  \partial_\lambda[E(\lambda)(r_1)] \ \widetilde\chi(\lambda r_1) \chi(\lambda)   \lambda^{n-1}  F_\pm(\lambda r_2)  \ d\lambda \\
 		&+ \frac1{r_2} \int_0^\infty  e^{\pm  i  \lambda   r_2 }  E(\lambda)(r_1)\ \partial_\lambda \big[ \widetilde\chi(\lambda r_1) \chi(\lambda)   \lambda^{n-1}  F_\pm(\lambda r_2)  \big] \ d\lambda. \label{eq:K3 bad2}
 	\end{align}
 	For the second integral, \eqref{eq:K3 bad2},  we integrate by parts $N=\lceil \frac{n}2\rceil$ more times using  \eqref{Fbounds}, to obtain the bound
 	$$
 	\frac1{r_2^{N+\frac{n}2+\frac12}} \sum_{j_1+j_2 \leq N,  \ 0\leq j_1, j_2} \int_{r_1^{-1}}^1  \big| \partial_\lambda^{j_1} [E(\lambda)(r_1)]\big|  \      \lambda^{\frac{n-3}2-j_2 }   \ d\lambda.  
 	$$
 	Using Corollary~\ref{cor:E} we bound this by
 	$$
 	\les  \frac1{r_2^{N+\frac{n}2+\frac12 }}\Big[\int_{r_1^{-1}}^1  r_1^{2m-n}      \lambda^{\frac{n-3}2-N }   \ d\lambda +  \sum_{j_1+j_2 \leq N,  \ 0\leq j_1, j_2} \int_{r_1^{-1}}^1 r_1^{j_1+\frac{1-n}2}    \lambda^{n-2m -1- j_2 }   \ d\lambda\Big].
 	$$
 	The first integral takes care of the  additional term that arises in Corollary~\ref{cor:E} (for $\lambda r\gtrsim 1$)  in the case $j_1=0$.    Letting $\{n/2\}=  n/2-\lfloor n/2\rfloor$, we bound this by
 	
 	$$
 	\les \frac{r_1^{\{n/2\}+\frac12 +2m-n}+r_1^{\{n/2\}+\frac12}}{r_2^{n+\{n/2\}+\frac12}}
 	\les \frac{r_1^{\{n/2\}+\frac12}}{r_2^{n+\{n/2\}+\frac12}},
 	$$
 	whose contribution is admissable   by Lemma~\ref{lem:admissible2} since $r_2\gg \la r_1\ra$. 
 	
 	For the first integral, \eqref{eq:K3 bad1}, we integrate by parts $N=\lceil \frac{n}2\rceil$ more times after pulling out the phase $e^{i\lambda r_1}$ to obtain the bound
 	$$\frac1{r_2^{ \frac{n}2+\frac12} |r_1\pm r_2|^N} \sum_{j_1+j_2 \leq N,  \ 0\leq j_1, j_2} \int_{r_1^{-1}}^1   \big| \partial_\lambda^{j_1}[\widetilde E(\lambda)(r_1)] \big| \   \lambda^{\frac{n-1}2-j_2 } \ d\lambda 
 	$$
 	$$
 	\widetilde E(\lambda)(r_1):= e^{- i\lambda r_1} \partial_\lambda[E(\lambda)(r_1)]
 	$$
 	Using Corollary~\ref{cor:E}, we bound this by
 	\begin{multline*}
 		\frac1{r_2^{ n+\{n/2\}+\frac12}  } \sum_{j_1+j_2 \leq N,  \ 0\leq j_1, j_2} \int_{r_1^{-1}}^1  r_1 \frac{(\lambda r_1)^{\frac{n+1}2-2m}}{r_1^{n-2m}}    \lambda^{\frac{n-1}2-j_1-j_2 } \ d\lambda \\
 		\les \frac1{r_2^{ n+\{n/2\}+\frac12}  r_1^{\frac{n-3}2} }  \int_{r_1^{-1}}^1 \lambda^{\frac{n }2-2m -\{n/2\} } \ d\lambda
 		\les  \frac{1}{r_2^{n+\{n/2\}+\frac12}},
 	\end{multline*}
 	which is admissible by Lemma~\ref{lem:admissible}.
 \end{proof}

\section{Technical Lemmas}\label{sec:Minv}

It remains only to prove Lemma~\ref{lem:Gamma} stating  that the operators $\Gamma_k(\lambda)$ defined in \eqref{Gammalambda} satisfy the bounds needed to apply Proposition~\ref{lem:low tail low d}. This follows, with some modifications, from the discussion preceeding Lemma~3.5 in \cite{EGWaveOp}.  For the convenience of the reader, we sketch the argument here. In addition, we will state and prove two lemmas on admissable kernels that were used in the proof of Proposition~\ref{lem:low tail low d}. 

We write $n_{\star}$ to denote $n+4$ if $n$ is odd and $n+3$ if $n$ is even. The bounds in Lemma~\ref{prop:F} and Corollary~\ref{cor:E} imply that the operator $R_j$ with kernel
\be\label{R_k}
R_j(x,y):=v(x)v(y)\sup_{0<\lambda<1}|\lambda^{\max(0,j-1)} \partial_\lambda^j \mR_0^+(\lambda^{2m})(x,y)|
\ee
satisfies 
$$
R_j(x,y) \les v(x) v(y) 
\big(|x-y|^{2m+1-n}+|x-y|^{j-(\frac{n-1}{2})}\big),\,\,\,j\geq 1,
$$
$$
R_0(x,y) \les v(x) v(y) 
\big(|x-y|^{2m-n}+|x-y|^{-(\frac{n-1}{2})}\big).
$$
Therefore, $R_j$ is bounded on $L^2(\R^n)$ for $0\leq j\leq \lceil \frac{n}2\rceil+1$ provided that $|V(x)|\les \la x \ra^{-\beta}$ for some $\beta>n_{\star}$.  Indeed,  when $n<4m$, it follows because $R_j$ is Hilbert-Schmidt. Also the second term in the bounds above is always Hilbert-Schmidt. When $n\geq 4m $, we identify $|x-y|^{2m-n}$ (similarly $|x-y|^{2m+1-n}$) as a multiple of the fractional integral operator $I_{2m}:L^{2,\sigma}\to L^{2,-\sigma}$.  
Using the decay of $v(x)v(y)$ and identifying $\sigma =\frac{\beta}{2}$ suffices to apply the Propositions~3.2 and 3.3 in \cite{GV} and establish boundedness on $L^2$. 

Similarly, $\mathcal E(\lambda):=v[\mR^+_0(\lambda^{2m})-\mR^+_0(0)]v$ satisfies by the discussion above and Corollary~\ref{cor:E} that
$$
\|\mathcal E(\lambda)\|_{L^2\to L^2} \les \lambda.
$$

	Now, we define the operator
	$$
	T_0:=U+v\mR_0^+(0)v=M^+(0).
	$$
	By the assumption that zero energy is regular, $T_0$ is invertible, see e.g. \cite{soffernew}. Note that by a Neumann series expansion and the invertibility of $T_0$ we have 
	$$
	[M^+(\lambda)]^{-1} =\sum_{k=0}^\infty (-1)^k T_0^{-1} (\mE(\lambda)T_0^{-1})^{k}.
	$$
	The series converges in the operator norm on $L^2$ for sufficiently small $\lambda$. 
 By the resolvent identity the operator
	$ \partial_\lambda^N[M^+(\lambda)]^{-1} $ is a linear combination of operators of the form 
	$$
	[M^+(\lambda)]^{-1} \prod_{j=1}^J\big[v\big( \partial_\lambda^{N_j}\mR_0^{+}(\lambda^{2m})\big)v[M^+(\lambda)]^{-1}\big],
	$$
	where $\sum N_j=N$ and each $N_j\geq 1$. From the discussion above on $R_j$'s  this representation implies  that  
	\be\label{eq:MGamma}
 \sup_{0<\lambda<\lambda_0}\lambda^{\max(0,N-1)} | \partial_\lambda^N[M^+(\lambda)]^{-1}(x,y)|
	\ee
	is bounded in $L^2$ for $N=0,1,\ldots,\lceil \frac{n}{2}\rceil+1$ provided that $\beta >n_\star$.   

Recalling the definition of $\Gamma_k(\lambda)$, \eqref{Gammalambda}, and noting the $L^2$ boundedness of $R_j$'s above we see that
$$ \sup_{0<\lambda<\lambda_0}\lambda^{\max(0,N-1)} \big| \partial_\lambda^N \big(Uv\mR_0^+(\lambda^{2m}) \big(V\mR_0^+(\lambda^{2m})\big)^{k-1} v \big)(x,y)\big|
$$
is bounded on $L^2$. This yields   Lemma~\ref{lem:Gamma}	when $2m<n<4m$. 

When $n\geq 4m$ we need stronger bounds on the kernel of $\Gamma_k(\lambda)$.   We write the iterated resolvents 
\begin{align}\label{eq:Alambda}
	A(\lambda, z_1,z_2) =  \big[ \big(\mR_0^+(\lambda^{2m})  V\big)^{k-1}\mR_0^+(\lambda^{2m})\big](z_1,z_2).
\end{align}  
For odd $n>4m$.  If   $ k-1 $ is sufficiently large depending on $n,m$ and $|V(x)|\les \la x \ra^{-\frac{n_\star}2-}$, then 
\begin{align*}
	\sup_{0<\lambda <1}| \lambda^{\max\{0,\ell-1\}} \partial_\lambda^\ell 	A(\lambda, z_1,z_2)|&\les \la z_1 \ra^2  \la z_2\ra^2,
\end{align*}
for $0\leq \ell\leq \frac{n+3}{2}=\lceil \frac{n}2\rceil+1$. This follows from the pointwise bounds on $R_j$ above. The iteration of the resolvents smooths out the local singularity $|x-\cdot|^{2m-n}$.  Each iteration improves the local singularity by $2m$, so that after $j$ iterations the local singularity is of size $|x-\cdot|^{2mj-n}$.  Selecting $k-1$ large enough ensures that the local singularity is completely integrated away. See \cite{EGWaveOp} for more details. For even $n\geq 4m$ we get a better bound since we need fewer derivatives: 
$$	\sup_{0<\lambda <1}|\lambda^{\max(0,\ell-1)} \partial_\lambda^\ell 	A(\lambda, z_1,z_2)| \les \la z_1 \ra^{\frac32}  \la z_2\ra^{\frac32},
$$
for $0\leq \ell\leq \frac{n+2}{2}=\lceil \frac{n}2\rceil+1$.

	Finally, recalling that  
	$$
		\Gamma_k(\lambda)=Uv A(\lambda)vM^{-1}(\lambda)v A(\lambda)vU
	$$
	yields Lemma~\ref{lem:Gamma} when $n\geq 4m$.

The following lemmas were used frequently  in the proof of  Proposition~\ref{lem:low tail low d}:
\begin{lemma}\label{lem:admissible}
	Let $K$ be an operator with integral kernel $K(x,y)$ that satisfies the bound
	$$
		|K(x,y)|\les \int_{\R^{2n}} \frac{v(z_1)v(z_2) \widetilde \Gamma(z_1,z_2) \chi_{\{|y-z_2|\gg \la z_1-x\ra \} }}{|x-z_1|^{n-2m-k} |z_2-y|^{n+\ell} } \, dz_1 \, dz_2
	$$
	for some $0\leq k\leq n-2m$ and $\ell>0$.  Then, under the hypotheses of Lemma~\ref{lem:low tail low d}, the kernel of $K$ is admissible, and consequently $K$ is a bounded operator on $L^p(\R^n)$ for all $1\leq p\leq \infty$.
	
\end{lemma}

\begin{proof}
	
	We first consider integration in $y$,
	\begin{align*}
		\int_{\R^n}|K(x,y)|\, dy \les 
		\int_{\R^{2n}}\frac{v(z_1)v(z_2) \widetilde \Gamma(z_1,z_2) \chi_{\{|y-z_2|\gg \la z_1-x\ra \} }}{|x-z_1|^{n-2m-k}  }   \int_{\R^n} \frac{\chi_{\{|y-z_2|\gg \la z_1-x\ra \} }}{|z_2-y|^{n+\ell}} \, dy\, \, dz_1 \, dz_2.
	\end{align*}
	Writing the $y$ integral in polar coordinates centered at $z_2$, and noting that $|z_2-y|\gtrsim 1$, we bound this by
	\begin{align*}
		\int_{\R^n}|K(x,y)|\, dy &\les 
		\int_{\R^{2n}}\frac{v(z_1)v(z_2) \widetilde \Gamma(z_1,z_2) }{|x-z_1|^{n-2m-k}  } \int_{1}^\infty r^{-1-\ell}  \, dr  \, dz_1 \, dz_2\\
		&\les \int_{\R^{2n}}\frac{v(z_1)v(z_2) \widetilde \Gamma(z_1,z_2) }{|x-z_1|^{n-2m-k}  } \, dz_1 \, dz_2
	\end{align*}
	If $2m<n<4m$, then the singularity in $z_1$ is locally $L^2$ and one bounds this as
	$$
		\int_{\R^{n}}|K(x,y)|\, dy \les \|v(\cdot) |x-\cdot|^{k+2m-n}\|_2 \| \widetilde \Gamma \|_{2\to 2} \| v(z_2)\|_2 \les 1,
	$$
	uniformly in $x$.  If $n\geq 4m$, one has
	$$
		\int_{\R^n}|K(x,y)|\, dy \les \int_{\R^{2n}}\frac{v(z_1)v(z_2)  \la z_1\ra^{-\frac{n}{2}-} \la z_2\ra^{-\frac{n}{2}-} }{|x-z_1|^{n-2m-k}  } \, dz_1 \, dz_2 \les  \la x\ra^{k+2m-n} \les 1,
	$$
	uniformly in $x$ since $k+2m-n\leq 0$.
	
	Next, integration in $x$ follows identically when $|x-z_1| \gtrsim 1$ noting that since $|y-z_2| \gg |z_1-x|$ we have
	$$
		\frac{1}{|x-z_1|^{n-2m-k}|z_2-y|^{n+\ell}  } \leq \frac{1}{|x-z_1|^{n+\ell}  |z_2-y|^{n-2m-k}  }.
	$$
	If $|x-z_1|<1$, we use polar co-ordinate in $x$ centered at $z_1$ to bound with
	\begin{multline*}
		\int_{\R^n} |K(x,y)|\,dx \les \int_{\R^{2n}} \frac{v(z_1)v(z_2) \widetilde \Gamma(z_1,z_2) \chi_{\{|y-z_2|\gg 1 \} }}{ |z_2-y|^{n+\ell} }  \int_0^1 r^{2m+k-1}\, dr\, \, dz_1 \, dz_2\\ 
		\les \int_{\R^{2n}}   v(z_1)v(z_2) \widetilde \Gamma(z_1,z_2)   \, dz_1 \, dz_2 .
	\end{multline*}
	Which is bounded uniformly in $y$.  
	
\end{proof}

We also need the following bound.
\begin{lemma}\label{lem:admissible2}
	
	Let $K$ be an operator with integral kernel $K(x,y)$ that satisfies the bound
	$$
	|K(x,y)|\les \int_{\R^{2n}} \frac{v(z_1)v(z_2) \widetilde \Gamma(z_1,z_2) \chi_{\{|y-z_2|\gg \la z_1-x\ra \}} |x-z_1|^{\ell} }{  |z_2-y|^{n+\ell} } \, dz_1 \, dz_2
	$$
	for some $\ell>0$.  Then, under the hypotheses of Lemma~\ref{lem:low tail low d}, the kernel of $K$ is admissible, and consequently $K$ is a bounded operator on $L^p(\R^n)$ for all $1\leq p\leq \infty$.
\end{lemma} 
 \begin{proof} 
Without loss of generality, we may assume that $|x-z_1|>1$.  If not, we may bound $|x-z_1|^\ell \les 1$ and apply Lemma~\ref{lem:admissible}.  We consider the $y$ integral first and use polar co-ordinates centered at $z_2$ to see
	\begin{multline*}
		\int_{\R^n}|K(x,y)| \, dy\les \int_{\R^{2n}} v(z_1)v(z_2) \widetilde \Gamma(z_1,z_2)|x-z_1|^{\ell}   \int_{|x-z_1|}^\infty r^{-1-\ell} \, dz_1 \, dz_2\\
		\les \int_{\R^{2n}} v(z_1)v(z_2) \widetilde \Gamma(z_1,z_2)  \, dz_1 \, dz_2 \les 1.
	\end{multline*}
	The bound holds uniformly in $x$. 
	
	For the $x$ integral, we use polar co-ordinates centered at $z_1$ to see
	\begin{multline*}
		\int_{\R^n}|K(x,y)| \, dx\les \int_{\R^{2n}} \frac{v(z_1)v(z_2) \widetilde \Gamma(z_1,z_2) \chi_{\{|y-z_2|\gg \la z_1-x\ra \}}   }{  |z_2-y|^{n+\ell} }  \int_{0}^{|z_2-y|} r^{n+\ell-1} \,dr \, dz_1 \, dz_2\\
		\les \int_{\R^{2n}} v(z_1)v(z_2) \widetilde \Gamma(z_1,z_2)  \, dz_1 \, dz_2 \les 1,
	\end{multline*}
	uniformly in $y$.
	
\end{proof}

\end{document}